\documentclass[preprint,review,12pt]{elsarticle}
\usepackage{amsmath,amssymb,amsfonts,amsthm}
\usepackage{latexsym,bm}
\usepackage{graphicx}
\usepackage{float}
\usepackage{stmaryrd}

\usepackage{background}
\SetBgContents{}

\usepackage[CJKbookmarks,bookmarks=true,bookmarksnumbered=true]{hyperref}
\hypersetup{pdfsubject={StablizedDGTheory12}, pdfauthor={GZH-CJW of HENU}}

\usepackage{natbib}
\biboptions{numbers,sort&compress}

\journal{ NA}

\begin{document}

\begin{frontmatter}

\title{A new variational formulation based on discontinuous Galerkin technique for a reaction-diffusion problem\tnoteref{1}}\tnotetext[1]{The
work is supported by the Natural Science Foundation of China(No. 10901047).\\
 Email: zhihaoge@henu.edu.cn, caojiwei666@sina.com, tel:+86-13663786282, fax:+86-378-3881696.}
\author{Zhihao G${\rm e^{1,2}}$,\ Jiwei Ca${\rm o^1}$}
\address{$ ^1$School of Mathematics and Information Sciences,
 Henan University, Kaifeng 475004, P.R. China\\
$ ^2$Institute of Applied Mathematics,
 Henan University, Kaifeng 475004, P.R. China}

\begin{abstract}
In this paper, a new variational formulation based on discontinuous Galerkin technique for a reaction-diffusion problem is introduced, and the discontinuous Galerkin technique of this work is different from the general discontinuous Galerkin methods. The well posedness of the new formulation is given. Finally, it is pointed that the new variational formulation will be helpful to design better hybrid numerical methods which  will not only strongly stable in spatial variable and absolutely stable in temporal variable but also be optimally convergent.
\end{abstract}


\end{frontmatter}

\newtheorem{thm}{Theorem}[section]
\newtheorem{lem}{Lemma}[section]
\newtheorem{cor}{Corollary}[section]
\newtheorem{prop}{Proposition}[section]
\newtheorem{remark}{Remark}[section]
\renewcommand{\theequation}{\thesection.\arabic{equation}}
\numberwithin{equation}{section}

\section{Introduction}
\setcounter{equation}{0}
In this work, we propose a new variational formulation based on discontinuous Galerkin technique for a reaction-diffusion problem within a new function space setting, the reaction-diffusion problem is
\begin{eqnarray}
- \nabla \cdot (K(x)\nabla u) + u &=& f,\ \textrm{in}\ \Omega,\label{eq3}\\
             u &=& 0,\ \textrm{on}\ \partial \Omega,\label{eq303}
\end{eqnarray}
where $f$ is a real-valued function in $L^2(\Omega)$ and
$0 < K_0 \leq K(x) \leq K_1$.

The problem (\ref{eq3})-(\ref{eq303}) is an important and basic mathematical model, widely used in many fields. As for the theoretical result of the above model, one can see \cite{Gilbarg-Trudinger} and so on.

The first discontinuous Galerkin (DG) method for hyperbolic equations was introduced by \cite{[7]W.H.Reed},
and since that time there has been an active development of DG methods for hyperbolic and nearly hyperbolic problems,
resulting in a variety of different methods. For elliptic and parabolic
equations, discontinuous finite elements were proposed by many researchers, such as \cite{[3]V.Dolejsi, [4]D.N.ArnoldandF.Brezzl, [5]B.Riviere, [12]D.N.ArnoldandB.Cockburn,[13]I.BabuskaandC.Baumann,[28]B.RiviereandM.F.Wheeler, [10]J.Douglas,[8]I.Babuska,[9]I.BabuskaandM.Zlamal,[19]F.BassiandS.Rebay,[20]B.CockburnandC.W.Shu,[21]J.OdenandC.E.Baumann,[14]F.BrezziandG.Manzini,[2]A.Romkes}.

The key idea of the paper is to propose a new variational formulation based on discontinuous Galerkin technique, and the discontinuous Galerkin technique is different from the general discontinuous Galerkin methods.
The formulation satisfies a local conservation property, and we prove well posedness of the new formulation by proving and using inf-sup condition.

The paper is organized as follows. In Section 2, we introduce the new weak formulation of the problem (\ref{eq3})-(\ref{eq303}). In Section 3, we investigate the well posedness of
the variational formulation, which includes the continuity property of the bilinear form and the inf-sup condition.
Finally, some concluding remarks are summarized.

\section{A new variational formulation }
Let $\Omega \subset \mathbb{R}^2$ be a bounded open domain with Lipschitz boundary
$\partial \Omega$ and let $\{P_h\}$ be a family of regular partitions of $\Omega$
into open elements $E$ such that $\Omega = \textrm{int} ( \bigcup\limits_{E \in P_h} \bar{E} )$.

The following notations will be used in our further considerations. Denote $h=\max_{E \in P_h}h_E$, where
$h_E=\textrm{diam}(E)$. The set of all edges of the
partition $P_h$ is given by $\varepsilon_h = \{\gamma_k\}, k=1,\ldots,N_{edge}$,
where $N_{edge}$ reprents the number of edges in the partition $P_h$. The interior
interface $\Gamma_{\textrm{int}}$ is then defined as the union of all common edges shared
by elements of partition $P_h$, that is,
\begin{eqnarray}
\Gamma_{\textrm{int}} = \bigcup\limits_{k=1}^{N_{edge}}\gamma_e \backslash \partial \Omega.\nonumber
\end{eqnarray}


For the sake of clarity in the notation, the jump and average operators 
are defined by
{\newcommand\langlen{\ensuremath{\langle}}
\begin{eqnarray}
[v]=v|_{\gamma_e \subset \partial E_i}-v|_{\gamma_e \subset \partial E_j},
\langle v\rangle =\frac12(v|_{\gamma_e \subset \partial E_i}
+ v|_{\gamma_e \subset \partial E_j}),\ i>j,\label{120409-1}
\end{eqnarray}
where $\gamma_e=\textrm{int}(\partial E_i \cap \partial E_j)$ is
the common edge in 2D (or interface in 3D) between two neighbouring elements, see Figure \ref{figdoc666}.
\begin{figure}[H]
\centering
\includegraphics[width=6.5cm]{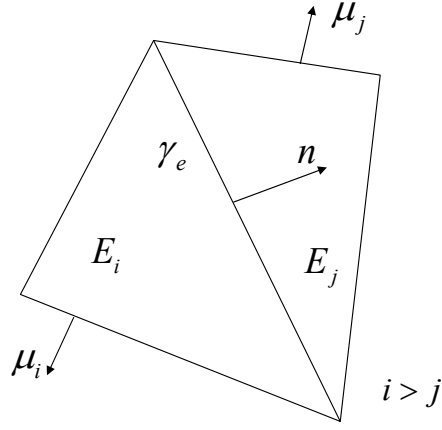}
\caption{Geometrical definitions of neighbouring elements.}\label{figdoc666}
\end{figure}

First, we introduce the following broken Sobolev space:
\begin{eqnarray}
\mathcal{M}(P_h)=\{v \in L^2(\Omega)|v\shortmid_E \in H(\Delta,E),\forall E \in P_h,
[\nabla v \cdot {\bm n}] \in L^2(\Gamma_{\textrm{int}})\},\nonumber
\end{eqnarray}
where
\begin{eqnarray}
H(\Delta,E)=\{v \in L^2(E)|\nabla \cdot \nabla v \in L^2(E)\} \subset H^1(E).\nonumber
\end{eqnarray}
Notice here, that $v \in H(\Delta,E)$ implies $\nabla v \cdot {\bm \mu}
\in H^{-1/2}(\partial \Omega)$.
The norm $||| \cdot |||$ on $\mathcal{M}(P_h)$ is defined as
{\setlength\arraycolsep{2pt}
\begin{eqnarray}\label{eq1}
|||v|||^2 &=& \sum_{E \in P_h}\Big\{\|v\|^2_{\ast}+\frac{h^\nu}{p^\theta}
\|K(x)\nabla v \cdot {\bm \mu}\|^2_{H^{-1/2}(\partial E)}\Big\} {} \nonumber\\
&& {} +\sigma\frac{h^\lambda}{p^\zeta}\|
[K(x)\nabla v \cdot {\bm \mu}]\|^2_{L^{2}(\Gamma_{\textrm{int}})} {}.
\end{eqnarray}}
where we denote that $\|v\|^2_{\ast}= \int_E |K(x)| |\nabla v|^2 dx +
\int_E |v|^2dx=\|K^{\frac{1}{2}} \nabla v\|^2_{L^2(E)}+\|v\|^2_{L^2(E)}$,
and one can easily prove that norms $\|\cdot\|_{\ast}$ and $\|\cdot\|_{H^1(E)}$
are equivalent. The parameter $p \in \mathbb{R}$ that is introduced here
represents the minimum of all of the local orders of polynomial approximations
$p_E$ in the partition $P_h$. The parameters $\nu,\lambda,\theta,\zeta$
are greater than or equal to zero and that the subsequent norms in (\ref{eq1})
are defined as
\begin{eqnarray}
\|u\|_{H^{-1/2}(\partial E)}&=&\sup_{\varphi \in H^{1/2}(\partial E)}
\frac{|{\langle u,\varphi \rangle}_{-1/2 \times 1/2,\partial E}|}
{\|\varphi\|_{H^{1/2}(\partial E)}},\label{eq8}\\
\|\varphi\|_{H^{1/2}(\partial E)}
&=&\inf_{\substack{w \in H^{1}(E) \\ \gamma_0 w =\varphi}} \|w\|_{\ast},\label{eq9}
\end{eqnarray}
where ${\langle \cdot,\cdot \rangle}_{-1/2 \times 1/2,\partial E}$ denotes the duality
pairing in $H^{-1/2}(\partial E) \times H^{1/2}(\partial E)$, namely,
\begin{eqnarray}
\langle u, v \rangle_{-1/2 \times 1/2,\partial E}
=\int_{\partial E}uv ds.
\end{eqnarray}
And  $\gamma_0$
denotes the trace operator
$$\gamma_0:H^1(E) \rightarrow H^{1/2}(\partial E).$$
Now, the choice for the space of test functions, $V$, is the
completion of $\mathcal{M}(P_h)$ with respect to the norm $|||\cdot|||$.

The new discontinuous variational formulation, within this
new function space setting, is then stated as follows:
\begin{eqnarray}\label{eq2}
\textrm{Find}\ u \in V, s.t., B(u,v)=L(v),\ \forall v \in V,
\end{eqnarray}
where the bilinear form $B(u,v)$ and linear form $L(v)$ are defined as
\begin{eqnarray}
B(u,v) &=& \sum_{E \in P_h}\Big\{\int_{E}\big(K(x)\nabla u \cdot \nabla v + uv\big)dx {}\nonumber\\
&& {} - \int_{\partial E}\big(v(K(x)\nabla u \cdot {\bm \mu})
      -(K(x)\nabla v \cdot {\bm \mu})u\big)ds\Big\} {} \nonumber\\
&& {} + \int_{\Gamma_{\textrm{int}}}\big(\langle v\rangle[K(x)\nabla u \cdot {\bm n}]
      -\langle u\rangle[K(x)\nabla v \cdot {\bm n}]\big)ds {} \nonumber\\
&& {} + \int_{\Gamma_{\textrm{int}}} \sigma \frac{h^{\lambda}}{p^{\zeta}}
[K(x)\nabla u \cdot {\bm n}][K(x)\nabla v \cdot {\bm n}]ds, {} \label{eq5} \\
L(v) &=& \int_{\Omega}fvdx.\label{eq18}
\end{eqnarray}

The formulation (\ref{eq2}) (or VBVP (\ref{eq2})) is closely related to the DG method formulation by Oden, Babuska
and Baumann \cite{[1]J.T.Oden}. In fact, choosing the subspace $\tilde{V}(P_h)$
of $V$ of function with fluxes $\nabla v \cdot {\bm n} \in L^2(\partial E)$,
and using the following identities:
\begin{eqnarray}\label{eq6}
&&\quad \sum_{E \in P_h}\int_{\partial E} v(K(x)\nabla u \cdot {\bm \mu})ds\nonumber\\
&&=\int_{\Gamma_{\textrm{int}}}[v(K(x)\nabla u \cdot {\bm n})]ds
+\int_{\partial \Omega}v(K(x)\nabla u \cdot {\bm n})ds,
\end{eqnarray}
and
\begin{eqnarray}\label{eq7}
[v(K(x)\nabla u \cdot {\bm n})]=\langle K(x)\nabla u \cdot {\bm n} \rangle[v]
+\langle v \rangle [K(x)\nabla u \cdot {\bm n}],
\end{eqnarray}
we can get the DG formulation of \cite{[1]J.T.Oden}.
The only difference would then be the addition of the last term in (\ref{eq5}).
This term has been incorporated in \cite{[16]P.PercellandM.F.Wheeler,[22]T.J.R.HughesandG.Engel},
where it is accompanied by the jumps of the
function $[v]$ across the element interfaces. We replace the $[v]$ jumps by the
$[\nabla v\cdot {\bm \mu}]$ jumps, in order to prove both continuity and Inf-Sup
properties of the bilinear form with respect to the space $V$, in which
the norm is defined as $|||\cdot|||$.


\section{Well posedness of the new variational formulation}
In this section, we establish the well posedness of the variational formulation
(\ref{eq2}). Thus, we show that the solution of the problem (\ref{eq3})-(\ref{eq303})
is also a solution to the weak problem. And we prove the existence, uniqueness
and the continuous dependence on the input data of the solution to the variational formulation
(\ref{eq2}). Essential in some of these proofs are the continuity
inf-sup conditions of the bilinear form  (\ref{eq5}).

Now, we introduce an important lemma \cite{[2]A.Romkes} as follows.
\begin{lem}\label{lem6}
If $u \in H(\Delta,\Omega)$, then $u$ and $(\nabla u \cdot {\bm n})$ are weakly
continuous across the element interface $\Gamma_{{\rm int}}$ in the sense that
\begin{eqnarray}
\int_{\gamma_e}[u]\varphi ds=0,\qquad \int_{\gamma_e}
[\nabla u \cdot {\bm n}]\varphi ds=0,
\end{eqnarray}
where $\varphi$ belonges to $D(\gamma_e)$, and $\gamma_e= {\rm int}(\partial E_i \in\partial E_j)
\subset \Gamma_{{\rm int}}$.
\end{lem}

\begin{thm}\label{thm1}
Let $u$ be the solution of the problem (\ref{eq3})-(\ref{eq303}). Then $u$ is a solution to the variational formulation
(\ref{eq2}) as well.
\end{thm}

\begin{proof}
If we restrict (\ref{eq3}) to an element $E \in P_h$, multiply this local
equation by a test function $\varphi_E \in H^2(E)$, integrate over the element $E$, and
apply Green's identity, we get
\begin{eqnarray}
\int_E \big(K(x)\nabla u \cdot \nabla \varphi_E +u \varphi_E\big)dx
-\int_{\partial E}\varphi_E(K(x)\nabla u \cdot {\bm \mu})ds
=\int_E f \varphi_Edx.
\end{eqnarray}
Repeating this for all $E \in P_h$, extending each $\varphi_E$ to zero outside of $E$,
and summing in $E$, yields
\begin{eqnarray}
&&\quad \sum_{E \in P_h}\Big\{\int_E \big(K(x)\nabla u \cdot \nabla \varphi +u \varphi\big)dx
-\int_{\partial E}\varphi(K(x)\nabla u \cdot {\bm \mu})ds\Big\}\nonumber\\
&&=\sum_{E \in P_h}\int_E f \varphi dx.
\end{eqnarray}
where
$$\varphi = \sum_{E \in P_h} \varphi_E \in
\{v \in L^2(\Omega)|\, v|_E \in H^2(E),\forall E \in P_h\}.$$

Since $u$ is the solution of the problem (\ref{eq3})-(\ref{eq303}), we know that $u$
satisfies the Dirichlet boundary condition on $\partial \Omega$. In addition,
it is known that $u$ belongs to $H(\Delta,\Omega)$. From Lemma \ref{lem6},
we know that $u$ and $(\nabla u \cdot {\bm n})$ are weakly continuous
across the element interfaces to the variational formulation in a weak sense,
which yields
\begin{eqnarray}
&&\quad \sum_{E \in P_h}\Big\{\int_E \big(K(x)\nabla u \cdot \nabla \varphi +u \varphi\big)dx
-\int_{\partial E}\varphi(K(x)\nabla u \cdot {\bm \mu})ds\Big\}\nonumber\\
&&\quad +\int_{\Gamma_{\textrm{int}}}{\langle \varphi \rangle}[K(x)\nabla u \cdot {\bm n}]ds
+\int_{\Gamma_{\textrm{int}}}[u]{\langle K(x)\nabla \varphi \cdot {\bm n} \rangle}ds\nonumber\\
&&\quad +\int_{\partial \Omega} u(K(x)\nabla \varphi \cdot {\bm \mu})ds
+\int_{\Gamma_{\textrm{int}}}\sigma \frac{h^\lambda}{p^\zeta}
[K(x)\nabla u \cdot {\bm n}][K(x)\nabla \varphi \cdot {\bm n}]ds\nonumber\\
&&=\int_{\Omega}f \varphi dx,\quad \forall \varphi \in H^2(P_h).
\end{eqnarray}
Combining (\ref{eq6}) and (\ref{eq7}) gives
\begin{eqnarray}
&&\quad\sum_{E \in P_h} \Big\{\int_{E}\big(K(x)\nabla u \cdot \nabla \varphi +u \varphi\big)dx\nonumber\\
&&\quad-\int_{\partial E}\big(\varphi (K(x)\nabla u \cdot {\bm \mu})
-u(K(x)\nabla \varphi \cdot {\bm \mu})\big)ds\Big\}\nonumber\\
&&\quad+\int_{\Gamma_{\textrm {int}}}\big(\langle \varphi \rangle[K(x)\nabla u \cdot {\bm n}]
-\langle u \rangle[K(x)\nabla \varphi \cdot {\bm n}]\big)ds\nonumber\\
&&\quad+\int_{\Gamma_{\textrm {int}}}\sigma \frac{h^\lambda}{p^\zeta}
[K(x)\nabla u \cdot {\bm n}][K(x)\nabla \varphi \cdot {\bm n}]ds,\nonumber\\
&&=\int_{\Omega}f \varphi dx,\quad \forall \varphi \in H^2(P_h).
\end{eqnarray}
Applying the density of $H^2(P_h)$ in $V$, we complete the proof.
\end{proof}

\subsection{Continuity property}
\begin{thm}\label{thm5}
Let $B(\cdot,\cdot)$ be the bilinear form as defined in (\ref{eq5}).
If $\sigma >0$, then there exists $M>0$ such that
\begin{eqnarray}
|B(u,v)|\leq M|||u|||\,|||v|||,\ \forall u,v \in V,
\end{eqnarray}
\end{thm}

\begin{proof}
By the definition of the average, we can obtain
\begin{eqnarray}
\int_{\Gamma_{\textrm {int}}} \langle v \rangle[K(x)\nabla u \cdot {\bm n}]ds
=\frac12 \sum_{E \in P_E} \int_{\partial E \cap \Gamma_{\textrm {int}}}
v[K(x)\nabla u \cdot {\bm n}]ds,\ u,v \in V.
\end{eqnarray}
And by the definition of $B(\cdot,\cdot)$, then we get
\begin{eqnarray}
B(u,v) &=& \sum_{E \in P_h}\Big\{\int_{E}\big(K(x)\nabla u \cdot \nabla v + uv\big)dx {}\nonumber\\
&& {} - \int_{\partial E}v(K(x)\nabla u \cdot {\bm \mu})ds
      + \int_{\partial E}u(K(x)\nabla v \cdot {\bm \mu})ds {} \nonumber\\
&& {} + \frac12 \int_{\partial E \cap \Gamma_{\textrm {int}}}
        v[K(x)\nabla u \cdot {\bm n}]ds
      -\frac12 \int_{\partial E \cap \Gamma_{\textrm {int}}}
        u[K(x)\nabla v \cdot {\bm n}]ds\Big\} {}\nonumber\\
&& {} + \sigma \frac{h^{\lambda}}{p^{\zeta}}\int_{\Gamma_{\textrm{int}}}
[K(x)\nabla u \cdot {\bm n}][K(x)\nabla v \cdot {\bm n}]ds. {}
\end{eqnarray}
Applying the Schwarz inequality, (\ref{eq8})
and (\ref{eq9}), we get
\begin{eqnarray*}
&&\qquad B(u,v)\leq \max \Bigg\{3,\frac{Cp^\theta}{h^\nu},
\frac{Cp^\zeta}{4\sigma h^\lambda}+1\Bigg\}\\
&& \cdot \Big\{\sum_{E \in P_h}\Big(\|u\|^2_{\ast}
+\frac{h^\nu}{p^\theta}\|K(x)\nabla u \cdot {\bm \mu}\|^2_{H^{-1/2}(\partial E)}\Big)\\
&&\qquad +\sigma\frac{h^\lambda}{p^\zeta}
\|[K(x)\nabla u \cdot {\bm n}]\|^2_{L^2(\Gamma_{\textrm {int}})}\Big\}^{1/2}\\
&& \cdot \Big\{\sum_{E \in P_h}\Big(\|v\|^2_{\ast}
+\frac{h^\nu}{p^\theta}\|K(x)\nabla v \cdot {\bm \mu}\|^2_{H^{-1/2}(\partial E)}\Big)\\
&&\qquad +\sigma\frac{h^\lambda}{p^\zeta}
\|[K(x)\nabla v \cdot {\bm n}]\|^2_{L^2(\Gamma_{\textrm {int}})}\Big\}^{1/2}.
\end{eqnarray*}
Denote $M=\max \Bigg\{3,\frac{Cp^\theta}{h^\nu},
\frac{Cp^\zeta}{4\sigma h^\lambda}+1\Bigg\}$, where $C=\max\{\frac{1}{K_0},1\}$, hence the proof is completed.
\end{proof}
\subsection{The Inf-sup condition}
\subsubsection{The auxiliary problems}
Given an arbitrary $u \in V$, find for every $E \in P_h$ the function $z_E$, such that
\begin{eqnarray}\label{eq11}
\begin{split}
-\nabla \cdot (K(x)\nabla z_E) +z_E &= 0, &&\textrm{in}\ E,\\
K(x)\nabla z_E \cdot {\bm \mu} &= K(x)\nabla u \cdot {\bm \mu}, &&\textrm{on}\ \partial E.
\end{split}
\end{eqnarray}
The equivalent variational formulation of (\ref{eq11}) is that given $u \in V$,
find $z_E \in H^1(E)$ such that
\begin{eqnarray}
(z_E,v)_{\ast,E} = \int_{\partial E}(K(x)\nabla u \cdot {\bm \mu})\gamma_0 vds,
\quad v \in H^1(E),\label{eq10}
\end{eqnarray}
where $(\cdot,\cdot)_{\ast,E}$ denotes the inner product in $H^1(E)$, and it can
be proved that $(v,v)_{\ast,E}=\|v\|^2_{\ast}$ easily. By the generalized
Lax-Milgram\ theorem it follows that the problem (\ref{eq10}) has a unique solution
$z_E \in H^1(E)$.
\begin{remark}\label{remark1}
Substituting $z_E$ and $u$ for $v$ in (\ref{eq10}),
we obtain the following two identities:
\begin{eqnarray*}
\|z_E\|^2_{\ast}&=&\int_{\partial E}(K(x)\nabla u \cdot {\bm \mu})\gamma_0 z_Eds,\\
(z_E,u)_{\ast,E}&=&\int_{\partial E}(K(x)\nabla u \cdot {\bm \mu})\gamma_0 uds.
\end{eqnarray*}
\end{remark}

From \cite{[2]A.Romkes,[6]F.Brezzi}, we know that the following result holds.
\begin{thm}\label{eq12}
Given $u \in V$, let $z_E=z_E(u)$ be the unique solution to (\ref{eq10}), then
the following relation holds:
\begin{eqnarray*}
\|z_E\|_{\ast}=\|K(x)\nabla u \cdot {\bm \mu}\|_{H^{-1/2}(\partial E)}.
\end{eqnarray*}
\end{thm}

\subsubsection{Inf-sup condition on the space V}
In this section, we prove that bilinear form $B(\cdot,\cdot)$ satifies
the Inf-Sup condition with respect to the norm $|||\cdot|||$, defined
by (\ref{eq1}). Let us introduce the extension operator
$\Psi_E:\ H^1(E) \rightarrow V$,
\begin{eqnarray*}
\quad \Psi_E(v_E)=\left\{
\begin{array}{ll}
v_E,\ &\textrm{in}\ E,\\
0,\ &\textrm{in}\ \Omega\backslash E.
\end{array}\right.
\end{eqnarray*}
Hence, given a function $u \in V$, we can solve (\ref{eq10}) for a set
of functions $z_E(u)$ and construct a function $\hat{u} \in V$, such that
\begin{eqnarray}\label{eq15}
\hat{u}=u + \beta \sum_{E \in P_h}\Psi_E(z_E),
\end{eqnarray}
where $\beta \in \mathbb{R}$.
\begin{lem}\label{lem1}
Given $u \in V$, then for every $\beta \in \mathbb{R}$ there exists a
strictly positive $\xi_1=\xi_1(h,p)$ such that
\begin{eqnarray*}
|||\hat{u}||| \leq \xi_1 |||u|||.
\end{eqnarray*}
\end{lem}
\begin{proof}
Substitution of the definition of $\hat{u}$ into (\ref{eq1}) and recalling
from (\ref{eq11}) that $K(x)\nabla z_E \cdot {\bm \mu}
=K(x)\nabla u \cdot {\bm \mu}$ on $\partial E$, we obtain
\begin{eqnarray*}
|||\hat{u}|||^2 &=& \sum_{E \in P_h} \Big\{\|u\|^2_{\ast}
+2(u,\beta z_E)_{\ast,E}+\|\beta z_E\|^2_{\ast}{}\\
&& {} +(1+\beta)^2 \frac{h^\nu}{p^\theta}
\|K(x)\nabla u \cdot {\bm \mu}\|^2_{H^{-1/2}(\partial E)}\Big\}{}\\
&& {} +\sigma \frac{h^\lambda}{p^\zeta}
(1+\beta)^2\|[K(x)\nabla u \cdot {\bm n}]\|^2_{L^2(\Gamma_{\textrm{int}})}.{}
\end{eqnarray*}
Using the Schwarz inequality, triangle inequality, Theorem \ref{eq12},
and Young's inequality as follows:
\begin{eqnarray}\label{eq31}
2(u,z_E)_{\ast,E} \leq \varepsilon \|u\|^2_{\ast}
+\frac{1}{\varepsilon}\|z_E\|^2_{\ast},\ \varepsilon > 0,
\end{eqnarray}
we obtain (here taking $\varepsilon =1$)
\begin{eqnarray*}
|||\hat{u}|||^2 &\leq& \sum_{E \in P_h} \Big\{2\|u\|^2_{\ast}
+\Big((1+\beta)^2 +\frac{2\beta^2 p^\theta}{h^\nu}\Big)\frac{h^\nu}{p^\theta}
\|K(x)\nabla u \cdot {\bm \mu}\|^2_{H^{-1/2}(\partial E)}\Big\}{}\\
&&{} +(1+\beta)^2\sigma \frac{h^\lambda}{p^\zeta}
\|[K(x)\nabla u \cdot {\bm n}]\|^2_{L^2(\Gamma_{\textrm{int}})}.{}
\end{eqnarray*}
Thus, the assertion holds with
\begin{eqnarray}
\xi_1 = \sqrt{\max\Big\{2,(1+\beta)^2
+ \frac{2\beta^2 p^\theta}{h^\nu}\Big\}}.\nonumber
\end{eqnarray}
\end{proof}

\begin{lem}\label{lem2}
Given $u \in V$, then there exists $\xi_2=\xi_2(\sigma,h,p)>0$ such that
\begin{eqnarray}
B(u,\hat{u}) \geq \xi_2|||u|||^2.\label{eq120414-1}
\end{eqnarray}
\end{lem}
\begin{proof}
By replacing $v$ by $\hat{u}$ in the definition of $B(u,v)$, and recalling
that $K(x)\nabla z_E \cdot {\bm \mu}=K(x)\nabla u \cdot {\bm \mu}$, we get
\begin{eqnarray*}
&&B(u,\hat{u}) = \sum_{E \in P_h}\Big\{\|u\|^2_{\ast}
+\beta(u,z_E)_{\ast,E}-\beta\int_{\partial E}z_E (K(x)\nabla u \cdot {\bm \mu})ds{}\\
&&{} +\beta \int_{\partial E}u(K(x)\nabla u \cdot {\bm \mu})ds\Big\}
+\beta \int_{\Gamma_{\textrm{int}}}\langle z_E \rangle [K(x)\nabla u \cdot {\bm n}]ds{}\\
&&{} -\beta\int_{\Gamma_{\textrm{int}}}\langle u \rangle[K(x)\nabla u \cdot {\bm n}]ds {}
+\sigma\frac{h^\lambda}{p^\zeta}(1+\beta)\|[K(x)\nabla u \cdot {\bm n}]\|^2_{L^2(\Gamma_{\textrm{int}})}.
\end{eqnarray*}
For simplicity, the traces $\gamma_0 z_E$ and $\gamma_0 u$ have been denoted as $z_E$ and $u$, respectively.
Now, using the identities given in Remark \ref{remark1}, we can rewrite
the above expression as follows
\begin{eqnarray}\label{eq14}
B(u,\hat{u}) &=& \sum_{E \in P_h}\Big\{\|u\|^2_{\ast}
+2\beta(u,z_E)_{\ast,E}-\beta\|z_E\|^2_{\ast}\Big\}{}\nonumber\\
&&{} +\beta \int_{\Gamma_{\textrm{int}}}\langle z_E \rangle [K(x)\nabla u \cdot {\bm n}]ds
-\beta\int_{\Gamma_{\textrm{int}}}\langle u \rangle[K(x)\nabla u \cdot {\bm n}]ds{}\nonumber\\
&&{}+ \sigma\frac{h^\lambda}{p^\zeta}(1+\beta)\|[K(x)\nabla u \cdot {\bm n}]\|^2_{L^2(\Gamma_{\textrm{int}})}.{}
\end{eqnarray}

If we take a closer look at the terms involving integrals over $\Gamma_{\textrm{int}}$,
we see that
\begin{eqnarray*}
&& \beta \int_{\Gamma_{\textrm{int}}}\langle z_E \rangle [K(x)\nabla u \cdot {\bm n}]ds
\geq -\frac{|\beta|}{4}\sum_{E \in P_h}\|z_E\|^2_{\ast}
-\frac{|\beta|}{2}\|[K(x)\nabla u \cdot {\bm n}]\|^2_{L^2(\Gamma_{\textrm{int}})},\\
&& -\beta \int_{\Gamma_{\textrm{int}}}\langle u \rangle [K(x)\nabla u \cdot {\bm n}]ds
\geq -\frac{|\beta|}{4}\sum_{E \in P_h}\|u\|^2_{\ast}
-\frac{|\beta|}{2}\|[K(x)\nabla u \cdot {\bm n}]\|^2_{L^2(\Gamma_{\textrm{int}})}.
\end{eqnarray*}
Now, back substitution of these two results into (\ref{eq14}), yields
\begin{eqnarray*}
B(u,\hat{u}) &\geq& \sum_{E \in P_h}\Big\{(1-\frac{|\beta|}{4})\|u\|^2_{\ast}
+2\beta(u,z_E)_{\ast,E}
-\big(\beta+\frac{|\beta|}{4}\big)\|z_E\|^2_{\ast}\Big\}{}\\
&&{}+\big(\sigma\frac{h^\lambda}{p^\zeta}(1+\beta)
-|\beta|\big)\|[K(x)\nabla u \cdot {\bm n}]\|^2_{L^2(\Gamma_{\textrm{int}})}.
\end{eqnarray*}
Using (\ref{eq31}), Cauchy-Schwarz inequality and Theorem \ref{eq12}, we obtain
\begin{eqnarray*}
B(u,\hat{u}) &\geq& \sum_{E \in P_h}\Big\{
\big(1-\varepsilon|\beta|-\frac{|\beta|}{4}\big)\|u\|^2_{\ast}{}\\
&&{}-\big(\beta+\frac{|\beta|}{4}+\frac{|\beta|}{\varepsilon}\big)
\|K(x)\nabla u \cdot {\bm \mu}\|^2_{H^{-1/2}(\partial E)}\Big\}{}\\
&&{}+\big(\sigma\frac{h^\lambda}{p^\zeta}(1+\beta)
-|\beta|\big)\|[K(x)\nabla u \cdot {\bm n}]\|^2_{L^2(\Gamma_{\textrm{int}})}.{}
\end{eqnarray*}
With $\beta < 0$, we have
\begin{eqnarray*}
B(u,\hat{u}) &\geq& \sum_{E \in P_h}\Big\{
\big(1-\varepsilon|\beta|-\frac{|\beta|}{4}\big)\|u\|^2_{\ast}{}\\
&&{}+\big(\frac{3|\beta|}{4}-\frac{|\beta|}{\varepsilon}\big)
\|K(x)\nabla u \cdot {\bm \mu}\|^2_{H^{-1/2}(\partial E)}\Big\}{}\\
&&{}+\big(\sigma\frac{h^\lambda}{p^\zeta}(1-|\beta|)
-|\beta|\big)\|[K(x)\nabla u \cdot {\bm n}]\|^2_{L^2(\Gamma_{\textrm{int}})}.{}
\end{eqnarray*}
The second term in the right hand side is only positive for $\varepsilon > 4/3$.
If we take $\varepsilon=2$, then we get
\begin{eqnarray*}
B(u,\hat{u})\geq \min \Big\{1-\frac94|\beta|,\frac{|\beta|p^{\theta}}{4h^{\nu}},
1-|\beta|-\frac{|\beta|p^{\zeta}}{\sigma h^{\lambda}}\Big\} |||u|||^2.
\end{eqnarray*}
It is clear that, given the parameters
$\sigma, \lambda, \zeta, \nu$ and $\theta$, we can always find a
coefficient $\beta$ such that there exists a $\xi_2(\sigma,h,p)$, denoted by
\begin{eqnarray}\label{eq17}
\xi_2 = \min \Big\{1-\frac94|\beta|,\frac{|\beta|p^{\theta}}{4h^{\nu}},
1-|\beta|-\frac{|\beta|p^{\zeta}}{\sigma h^{\lambda}}\Big\},
\end{eqnarray}
that satisfies the inequality (\ref{eq120414-1}).
\end{proof}

\begin{thm}\label{thm2}
Given $\sigma > 0$, there exists $\gamma = \gamma(\sigma,h,p)>0$
such that
\begin{eqnarray}
\sup_{v \in V\backslash\{0\}} \frac{|B(u,v)|}{|||v|||} \geq \gamma |||u|||,
\qquad \forall u \in V.
\end{eqnarray}
\end{thm}
\begin{proof}
By definition of the supremum, we can obtain
\begin{eqnarray}
\sup_{v \in V\backslash\{0\}} \frac{|B(u,v)|}{|||v|||} \geq
\frac{|B(u,\hat{u})|}{|||\hat{u}|||},\qquad \forall u \in V.
\end{eqnarray}
where $\hat{u}$ is defined by (\ref{eq15}). Next, by applying Lemmas
\ref{lem1} and \ref{lem2}, we obtain
\begin{eqnarray}
\sup_{v \in V\backslash\{0\}} \frac{|B(u,v)|}{|||v|||} \geq
\frac{|B(u,\hat{u})|}{|||\hat{u}|||} \geq
\frac{\xi_2(\sigma,h,p)}{\xi_1(\sigma,h,p)}|||u|||,
\qquad \forall u \in V.
\end{eqnarray}
Taking $\gamma = \xi_2/\xi_1$, we finish the proof.
\end{proof}

\begin{cor}\label{cor3}
If $\lambda = \nu = \theta = \zeta =0$, 
then the inf-sup coefficient $\gamma$ is a constant.
\end{cor}
\begin{proof}
For simplicity, we set $\sigma = 1$. Choosing $\beta = 4/10$ and using (\ref{eq17}),
we have $\xi_1 = \sqrt{288}/10$ and $\xi_2 = 1/10$,
respectively, and it follows that $\gamma = 1/\sqrt{288}$.
\end{proof}

\begin{cor}\label{cor2}
If $\lambda = \nu$ and $\theta = \zeta$, 
then
for $h^{\lambda}/p^{\zeta}<1$ the coefficient $\gamma$ is bounded
a constant $C>0$.
\end{cor}
\begin{proof}
We still set that $\sigma = 1$, but now we choose $\beta = 4h^{\nu}/10p^{\theta}$.
If we take $h^{\lambda}/p^{\zeta}<1$, we obtain the following inequalities from
(\ref{eq17})
$$\xi_1 \leq \frac{\sqrt{228}}{10},\quad \xi_2 \geq \frac{1}{10}.$$
Hence, we conclude that $\gamma \geq 1/\sqrt{228}$.
\end{proof}

\subsection{Existence and uniqueness}
\begin{lem}
If $f \in L^2(\Omega)$, then there exists a unique solution $w \in V$
to the  VBVP (\ref{eq2}) that is a solution to the problem (\ref{eq3})-(\ref{eq303}).
\end{lem}
\begin{proof}
First, we introduce the classical variational formulation of the model
problem (\ref{eq3})-(\ref{eq303}) in  $H^1_0(\Omega)$:
\begin{eqnarray}\label{eq19}
\textrm{Find}\ w \in H^1_0(\Omega),s.t.,\ A(w,v)=L(v),
\quad \forall v \in H^1_0(\Omega),
\end{eqnarray}
where $L(v)$ is defined by (\ref{eq18}) and the bilinear form $A(w,v): H^1(\Omega) \times H^1(\Omega) \rightarrow \mathbb{R}$ is defined by
\begin{eqnarray*}
 A(w,v)= \int_{\Omega}(K(x)\nabla w \cdot \nabla v + wv)dx.
\end{eqnarray*}
By the generalized Lax-Milgram theorem and by equivalence of this formulation to
the problem (\ref{eq3})-(\ref{eq303}), we know that if $f \in L^2(\Omega)$ there exists
a unique solution $w \in H^1_0(\Omega) \cap H(\Delta,\Omega) \subset V$ to
(\ref{eq19}) that satisfies the model problem in a distributional sense.
Consequently, by Theorem \ref{thm1}, we know that $w \in V$ is a solution to
the VBVP (\ref{eq2}) as well.

Thus, following the existence theory for the continuous variational formulation,
we can easily prove the existence of a solution to the VBVP (\ref{eq2}), and we omit the detail of the proof.
Also, we know that the solution is unique because the bilinear form $B(u,v)$ of (\ref{eq5}) is positive definite, i.e.,
\begin{eqnarray*}
B(v,v) = \sum_{E \in P_h}\|v\|^2_{\ast}
+\sigma\frac{h^\lambda}{p^\eta}
\|[K(x)\nabla v \cdot {\bm n}]\|^2_{L^{2}(\Gamma_{\textrm{int}})}>0,
\quad \forall v \in V\backslash \{0\}.
\end{eqnarray*}
\end{proof}

\subsection{Stability}
The last requirement to ensure well posedness of the weak formulation
is stability, i.e., the VBVP (\ref{eq2}) is continuously dependent of the solution
on the input data. Corresponding to Corollary \ref{cor3} and \ref{cor2},
we have two propositions \ref{prop1} and \ref{prop2} as follows.

\begin{prop}\label{prop1}
If 
$\sigma = 1, \lambda
= \zeta = 0$, and $\nu=\theta =0$, then the
solution to the VBVP (\ref{eq2}) depends continuously on the input data,
i.e., given a small perturbation $\delta f \in L^2(\Omega)$,
then there exists a unique perturbation $\delta u \in V$ such that
$$\|\delta u\|_{H^1(P_h)} \leq |||\delta u|||
\leq C\|\delta f\|_{L^2(\Omega)},$$
where C is a constant independent of $h$ and $p$.
\end{prop}
\begin{proof}
Let $\delta f \in L^2(\Omega)$ be a perturbation in the input data $f$.
Consequently, since the problem under consideration is linear, this leads
to a perturbation $\delta u \in V$ in the solution $u$, which satisfies
\begin{eqnarray}\label{eq20}
B(\delta u,v)= \int_{\Omega} \delta fvdx,\quad \forall v \in V.
\end{eqnarray}
Using Theorem \ref{thm2}, we have
\begin{eqnarray}
|||\delta u||| \leq \frac1 \gamma \sup_{v \in V\backslash\{0\}}
\frac{|B(\delta u,v)|}{|||v|||}.\label{eq120414-2}
\end{eqnarray}
Using (\ref{eq20}), (\ref{eq120414-2}) and the Cauchy-Schwarz inequality, we obtain
\begin{eqnarray}
|||\delta u||| \leq \frac1 \gamma \sup_{v \in V\backslash\{0\}}
\frac{|\int_{\Omega} \delta fvdx|}{|||v|||}
\leq \frac1 \gamma \|\delta f\|_{L^2(\Omega)}.\label{eq120414-3}
\end{eqnarray}
If $\sigma = 1, \lambda = \nu = \theta = \zeta = 0$, using Lemmas \ref{lem1} and \ref{lem2}, Theorem \ref{thm2}, Corollary \ref{cor3} and (\ref{eq120414-3}), we can take $C=\frac{1}{\gamma}=\sqrt{228}$ to complete the proof.
\end{proof}


\begin{prop}\label{prop2}
If $\lambda = \nu$, $\theta = \zeta$, and $h^{\lambda}/p^{\zeta}<1$,
then the solution $u$ to the VBVP (\ref{eq2})
depends continuously on the input data, i.e., given a perturbation $\delta f
\in L^2(\Omega)$, there exists a unique perturbation $\delta u \in V$ such
that
\begin{eqnarray*}
\|\delta u\|_{H^1(P_h)} \leq |||\delta u|||
\leq C\|\delta f\|_{L^2(\Omega)},
\end{eqnarray*}
where C is a constant independent of $h$ and $p$.
\end{prop}
\begin{proof}
Following the proof of Proposition \ref{prop1}, and using Corollary \ref{cor2}, we get
\begin{eqnarray}
\|\delta u\|_{H^1(P_h)} \leq \frac1 \gamma \|\delta f\|_{L^2(\Omega)}.
\end{eqnarray}
Therefore, the proof is completed.
\end{proof}

\section{Conclusion}
We propose a new variational formulation by a new discontinuous Galerkin technique for
a two-dimensional reaction-diffusion problem with Dirichlet boundary conditions.
Our preliminary study discovers that there is a strong and intimate connection between
the new variational formulation and discontinuous Galerkin methods for
the reaction-diffusion problem with discontinuous coefficient,
we intend to further explore this relation, in particular,
to make use of this relation to design better hybrid numerical
methods which hopefully will not only strongly stable in spatial variable and absolutely stable
in temporal variable but also be optimally convergent.



\end{document}